\RequirePackage{fix-cm}
\documentclass[smallextended]{svjour3}       
\smartqed  
\usepackage{graphicx}
\usepackage{color}
\usepackage{longtable}
\usepackage{subfiles}
\usepackage{multirow,array}
\usepackage[figuresright]{rotating}
\usepackage{amsfonts}
\usepackage{mathrsfs}
\usepackage{amsmath}
\usepackage{amssymb}
\usepackage{pdflscape}
\usepackage{listings}
\usepackage{subfigure}
\usepackage{algorithm}

\DeclareMathOperator*{\argmin}{\mathop{\arg \min}}
\usepackage{latexsym}
\newcommand{\inprod}[2]{\langle #1 , #2 \rangle}

\def\inprod#1#2{\langle#1, \, #2\rangle}

\def\nn{\nonumber}

\def\bdefi{\begin{definition}}
\def\edefi{\end{definition}}
\begin{document}

\title{An inexact proximal majorization-minimization Algorithm for remote sensing image stripe noise removal}
\subtitle{}

\titlerunning{An inexact PMM Algorithm for remote sensing image stripe noise removal}        

\author{Chengjing Wang \and Xile Zhao \and Qingsong Wang \and Zepei Ma \and Peipei Tang
}

\authorrunning{C.J. Wang \and X.L. Zhao \and Q.S. Wang \and Z.P. Wang \and P.P. Tang} 

\institute{Chengjing Wang \at School of Mathematics, Southwest Jiaotong University \email{renascencewang@hotmail.com} \and
  Xile Zhao \at School of Mathematical Sciences/Research Center for Image and Vision Computing, University of Electronic Science and Technology of China \email{xlzhao122003@163.com} \and
  Qingsong Wang \at Department of Mathematics, National University of Defense Technology \email{nothing2wang@hotmail.com} \and
  Zepei Ma \at School of Mathematics, Southwest Jiaotong University \email{2900675707@qq.com} \and
  Peipei Tang (corresponding author) \at
              Hangzhou City University, \email{tangpp@hzcu.edu.cn}
}

\date{Received: date / Accepted: date}

\maketitle

\begin{abstract}
The stripe noise existing in remote sensing images badly degrades the visual quality and restricts the precision of data analysis. Therefore, many destriping models have been proposed in recent years. In contrast to these existing models, in this paper, we propose a nonconvex model with a DC function (i.e., the difference of convex functions) structure to remove the strip noise. To solve this model, we make use of the DC structure and apply an inexact proximal majorization-minimization algorithm with each inner subproblem solved by the alternating direction method of multipliers. It deserves mentioning that we design an implementable stopping criterion for the inner
subproblem, while the convergence can still be guaranteed. Numerical experiments demonstrate the superiority of the proposed model and algorithm.
\end{abstract}

\section{Introduction}\label{sec:Introduction}

Nowadays, remote sensing images are playing an important role in the environment, agriculture,  biology, mineralogy, and so on. One may see \cite{Richards22}, \cite{ZhouCGLW22}, \cite{HuangTTZHCR22} for more details. Due to the limitation of sensors, remote sensing images are inevitably corrupted by noise, which deteriorates the quality of remote sensing images and hinders their subsequent applications (e.g., classification \cite{HuangH14} and detection \cite{DuZTZ13}). Remote sensing image denoising (RSID), which removes the noise while preserving the details and edges, is one of the fundamental problems in remote sensing image processing and analysis \cite{Richards22}.

In this work, we mainly  focus on the stripe noise scenario.  The degradation process \cite{BoualiL11}, \cite{ShenZ09},  can be mathematically formulated as
\begin{eqnarray}
\begin{aligned}
f(x, y) = u(x, y) + s(x, y),
\end{aligned}\label{additional_noise}
\end{eqnarray}
where $f (x, y), u (x, y)$ and $s (x, y) $ represent the degraded image from the detector, the potential fringeless image, and the fringe component at $(x, y) $, respectively. In order to discuss the numerical algorithm, the matrix vector form of (\ref{additional_noise}) can be rewritten as follows
\begin{eqnarray}
\begin{aligned}
f = u+ s,
\end{aligned}\label{target_function1}
\end{eqnarray}
where $f, u $ and $s $ represent the vectorized discrete versions of $f (x, y), u (x, y) $ and $s (x, y) $, respectively.
\begin{itemize}
\item[(1)] Directionality: consider the gradient of fringe components in horizontal and vertical directions. Here, we use the convex function 1-norm as the sparse regularization and give
$$R_1 (s) = \| \nabla_y s \|_{1,1},$$
where $\nabla_y $ denotes a linear first order difference operator in the vertical direction.

The horizontal gradient of the desired image $u $ should be smooth. Therefore, the one-way total change \cite{BoualiL11} of the desired image is considered to be
$$R_2 (s) = \| \nabla_x u  \|_1= \| \nabla_x f - \nabla_x s  \|_{1,1}, $$
where $\nabla_x $ denotes the first order difference operator in the horizontal direction.

\item[(2)] Structural properties: moreover, different from the random noise, the fringe component shows a special columnar structure. Therefore, the regular term \cite{ChenHDZW} is designed as
$$R_3(s) = \| s \|_{2,1},$$
where $\| s\|_{2,1}= \sum\limits_{j=1}^n \big(\sum\limits_{i=1}^m s_{i,j}^2\big)^{\frac{1}{2}}, \forall s \in \mathbb{R}^{m\times n} $.
\end{itemize}

Based on the above analysis, the stripe component has significant directionality and structure. By combining $R_ 1,R_ 2 $ and $R_3$, Chen et al. \cite{ChenHDZW} proposed the fringe noise removal model as follows
\begin{eqnarray}
\begin{aligned}
\label{target_function}
\min_{s}\; & \{\lambda_1 \| \nabla_y s \|_{1,1}
+ \lambda_2 \| \nabla_x f - \nabla_x s \|_{1,1}
+ \lambda_3 \| s \|_{2,1}\},
\end{aligned}
\end{eqnarray}
where $\lambda_1,\lambda_2 $ and $\lambda_3 $ are three positive regularization parameters that balance the three terms.

In order to simplify the sign, we abbreviate
$p_1(s)=\lambda_1 \|\nabla_y s \|_{1,1}$,
$p_2(u)=\lambda_2 \|\nabla_x f - \nabla_x u \|_{1,1}$ and
$p_3(v)=\lambda_3 \| v \|_{2,1}=\lambda_3 \| h(v)\|_1$ with
$(h(v))_i:=\| v(:,i) \|$. So (\ref{target_function}) is rewritten as follows
\begin{eqnarray}
\begin{aligned}
\label{target_fun}
\min_{s}\; & \{p_1(s) + p_2(s) + p_3(s)\}.
\end{aligned}
\end{eqnarray}

For the convenience of discussion, each stripe line is regarded as a column. If the stripes are horizontal, rotate them so that the stripe lines are vertical. When the fringe component is extracted from the fringe image, the final denoising image can be estimated by the following formula
$$u = f -s.$$

Fan and Li \cite{FanL2017} proposed a nonconvex sparsity function called SCAD as a surrogate of the $\ell_{0}$ function, which has been used
as the regularizer to avoid the model overfitting in high-dimensional statistical learning. It can achieve a sparse estimation with fewer measurements, faster convergence and furthermore is more robust against noises. In this paper, we apply the SCAD function to substitute the $\ell_{1}$ function part in the model \eqref{target_fun} in order to remove the noise more well. Hence the considered model becomes a nonconvex model. In the computational aspect, we adopt an inexact proximal majorization-minimization (PMM) algorithm to solve the nonconvex model with each inner subproblem solved by the dual alternating direction method (dADMM). Although Tang et al. \cite{TangWST} proposed the PMM algorithm to solve the nonconvex square-root-loss regression problems and the inner subproblem is also solved inexactly by the semismooth Newton method, the stopping criterion for the inner subproblem is not implementable and remains on the theoretical level. Tang, Wang and Jiang \cite{TangWJ} proposed the proximal-proximal majorization-minimization (PPMM) algorithm to solve the nonconvex rank regression problems with each inner subproblem solved by the semismooth Newton method. They proposed a stopping criterion for the inner subproblem, however, the convergence analysis is relatively easier since the considered problem is polyhedral and the inner subproblem is an unconstrained one. Numerical experiments demonstrate that the modified model and the proposed algorithm is more competitive.

The remaining parts of this paper are organized as follows. In Section \ref{sec:Preliminaries}, we introduce some basic concepts and preliminary results. In Section \ref{sec:Algorithm}, we set up the model problem and present the details of the algorithm. In Section \ref{sec:Convergence analysis}, we analyze the convergence of the proposed algorithm. In Section \ref{sec:Numerical experiments}, we implement the numerical experiments to compare our model with the original convex model and also compare our algorithm with the existing algorithms. We conclude our paper in Section \ref{sec:Conclusion}.

\section{Preliminaries}
\label{sec:Preliminaries}
In this section, we introduce some preliminaries that will be used in this paper.

Let $\mathbb{X}$ and $\mathbb{Y}$ be finite dimensional Hilbert spaces.
For any convex function $f:\mathbb{X} \to (-\infty,+\infty]$,
the conjugate function of $f$ is defined as
$$f^* (y):= \sup_{x\in \textrm{dom}(f)} \;\Big\{\langle y,x\rangle - f(x)\Big\}.$$

For a closed proper convex function
$f:\mathbb{X} \to (-\infty,+\infty]$
and a parameter $\sigma > 0$,
the proximal mapping $\textrm{Prox}_{\sigma f}$ is defined as
$$\textrm{Prox}_{\sigma f}(y):= \mathop{\arg \min}_x
\Big\{f(x) + \frac{1}{2\sigma} \|y-x\|^2\Big\},
\forall\, y \in \mathbb{X}.$$
The proximal mapping $\mbox{Prox}_{\sigma f}$ is single-valued and continuous with the following Moreau identity (see e.g.,\cite[Theorem 31.5]{Rockafellar70}) holds
\begin{eqnarray*}
\textrm{Prox}_{\sigma f}(x) + \sigma \textrm{Prox}_{\frac{1}{\sigma}f^{*}}(\frac{1}{\sigma}x)=x.
\end{eqnarray*}

A multifunction $\mathcal{F}:\mathbb{X}\rightrightarrows\mathbb{Y}$ is locally upper Lipschitz continuous at $x\in\mathbb{X}$ if there exist a parameter $\kappa$ which is independent of $x$ and a neighbourhood $\mathcal{U}$ of $x$ such that $\mathcal{F}(y)\subseteq\mathcal{F}(x)+\kappa\|y-x\|\mathbb{B}_{\mathbb{X}}$ holds for any $y\in\mathcal{U}$, where $\mathbb{B}_{\mathbb{X}}$ is a unit ball of the space $\mathbb{X}$. The multifunction $\mathcal{F}$ is said to be piecewise polyhedral if its graph $\mbox{gph}\mathcal{F}:=\{(x,y)\, |\, y\in\mathcal{F}(x)\}$ is the union of finitely many polyhedral convex sets.

The Kurdyka-{\L}ojasiewicz (K{\L}) property (see e.g., \cite{bolte-pauwels2016}) plays a central role in the convergence analysis. For further discussion, we introduce this concept below.
\begin{definition}
	A proper lower semicontinuous function $r:\mathbb{X}\rightarrow(-\infty,+\infty]$ is said to have the K{\L} property at $x\in\mbox{dom}(\partial r)$ if there exist $\eta\in(0,+\infty]$, a neighbour $\mathcal{U}$ of $x$ and a continuous concave function $\varphi:[0,\eta)\rightarrow [0,+\infty)$ satisfying\\
	(1) $\varphi(0)=0$;\\
	(2) $\varphi$ is continuous at 0 and continuously differentiable on $(0,\eta)$;\\
	(3) $\varphi'(s)>0$, for all $0<s<\eta$\\
	such that the K{\L} inequality $\varphi'(r(x')-r(x))\mbox{dist}(0,\partial r(x'))\geq 1$ holds for any $x'\in \mathcal{U}$ and $r(x)<r(x')<r(x)+\eta$.
	If $r$ satisfies the K{\L} property at each point of $\mbox{dom}(\partial r)$, then $r$ is called a K{\L} function.
\end{definition}

\section{Problem setup and the algorithm}
\label{sec:Algorithm}

In this paper, we consider the following nonconvex model problem by employing the SCAD function instead of (\ref{target_fun})
\begin{eqnarray}\label{eq:g-function}
\min_s \;\Big\{
g(s):=p_1(s) - q_1(s) + p_2(s)-q_2(s)+p_3(s)-q_3(s)
\Big\},
\end{eqnarray}
where $q_1(s)=q_{\alpha,\lambda_1} (\nabla_y s),
q_2(u)=q_{\alpha,\lambda_2} (\nabla_x u)$ and
$q_3(v)=q_{\alpha,\lambda_3} (h(v))$, here
\begin{eqnarray*}
q_{\alpha,\lambda}(x)=\sum_{i=1}^{n}q^{\rm scad}(x_{i};\alpha,\lambda),\
q^{\rm scad}(t;\alpha,\lambda)=\left\{\begin{array}{ll}0,&\mbox{if}\quad
|t|\leq \lambda,\\
\frac{(|t|-\lambda)^2}{2(\alpha-1)},&\mbox{if}\quad \lambda\leq
|t|\leq \alpha\lambda,\\
\lambda |t|-\frac{\alpha+1}{2}\lambda^2,&\mbox{if}\quad
|t|>\alpha\lambda.
\end{array}\right.
\end{eqnarray*}
In the above formula, we take $\alpha=3.7$ in both the theoretical analysis and the numerical implementation. Note that the function $q_{\alpha,\lambda}(x)$ is continuously differentiable with
\begin{eqnarray*}
\frac{\partial q_{\alpha,\lambda}(x)}{\partial x_{i}}=\left\{\begin{array}{ll}0,&\mbox{if}\ |x_{i}|\leq\lambda,\\
\frac{\mbox{sign}(x_{i})(|x_{i}|-\lambda)}{\alpha-1},&\mbox{if}\ \lambda<|x_{i}|\leq \alpha\lambda,\\
\lambda\mbox{sign}(x_{i}),&\mbox{if}\ |x_{i}|>\alpha\lambda.\end{array}\right.
\end{eqnarray*}

\subsection{The inexact PMM for the problem \eqref{eq:g-function}}
\label{subsec:PMM}

In this subsection, we present the framework of the inexact PMM. For simplicity, we write the problem \eqref{eq:g-function} as
\begin{eqnarray}\label{eq:G-function}
\min_s \;
\Big\{
P(s)-Q(s)
\Big\},
\end{eqnarray}
where $P(s)=p_1(s)+ p_2(s) + p_3(s)$, $Q(s)=q_1(s)+ q_2(s) + q_3(s)$.

Although problem \eqref{eq:G-function} is nonconvex, it is in a DC form. Thus, we naturally linearize the latter convex term $Q(s)$ to transform a nonconvex problem to a convex one, add a proximal term to guarantee the strong convexity, then adopt the sequential convexification approach to solve the problem. Now we present the algorithm as follows.

\begin{algorithm}
	\caption{The inexact PMM for the problem \eqref{eq:G-function}:}\label{alg:PMM}
	Choose $s^{0}$. For $k=0,1,2,\ldots$, iterate:
	\begin{description}
		\item[1.] Compute $s^{k+1}$
\begin{eqnarray}
\begin{aligned}
\label{subproblem}
s^{k+1}&\approx\mathop{\arg \min}_s \;
\left\{P(s) - Q(s^k) - \langle \nabla Q(s^k), s-s^k \rangle + \frac{1}{2 \widetilde{\sigma}_k} \|s - s^k\|^2
\right\}.
\end{aligned}
\end{eqnarray}
		\item[2.] If a stopping criterion is not satisfied, go to \textbf{Step 1}.
	\end{description}
\end{algorithm}

\subsection{The dADMM for the subproblem \eqref{subproblem}}

In this subsection, we focus on how to solve the subproblem \eqref{subproblem}. In the problem \eqref{subproblem}, for $q_i(x)\,(i=1,2,3)$, at the $k$th iteration we have $q_i(x) \geq q_i(x^k) + \langle g_i^k , x-x^k \rangle $, where $g_i^k := \nabla q_i (x^k)$. We denote $\widehat p_1(s) := p_1(s) + \frac{1}{2 \widetilde \sigma_{k}} \|s - s^{k}\|^2 $ and $\widehat p_3(v) :=\lambda_3 \|h(v) \| - \langle g_h^k, h(v)\rangle $, where $g_h^k := \nabla q_3(h(v^k))$.
By introducing two variables $u$ and $v$, the problem (\ref{subproblem}) can be equivalently rewritten as
\begin{eqnarray}
\begin{aligned}
\label{con:subproblem}
\min_{s,u,v} \,\;
&\widehat p_1(s) - \langle g_1^k,s \rangle + p_2(u) - \langle g_2^k,u \rangle + \widehat p_3(v) \\
\mbox{s.t.} \quad & f-s-u=0, \\
& s-v=0.
\end{aligned}
\end{eqnarray}
The Lagrangian function associated with the problem (\ref{con:subproblem}) is given by
\begin{eqnarray*}
l(s,u,v;x,y)
&=& \widehat p_1(s) - \langle g_1^k,s \rangle + p_2(u) - \langle g_2^k,u \rangle + \widehat p_3(v) + \langle x,u+s-f \rangle + \langle y,v-s \rangle.
\end{eqnarray*}
Then the Karush-Kuhn-Tucker (KKT) condition for the problem \eqref{con:subproblem} is
\begin{eqnarray}
&f-s-u=0,\ s-v=0,\ s-\textrm{Prox}_{\widehat{p}_{1}}(s+g^{k}_{1}-x+y)=0,&\nonumber\\
&u-\textrm{Prox}_{p_{2}}(u+g^{k}_{2}-x)=0,\ v-\textrm{Prox}_{\widehat{p}_{3}}(v-y)=0.&\label{eq:KKT-subproblem}
\end{eqnarray}
And the dual problem is
\begin{eqnarray}
\begin{aligned}
\label{con:equal_subproblem_dual}
\min_{x,y,z,\hat x,\hat y } \; &\widehat p_1^*(z) + p_2^*(\widehat x) + \widehat p_3^*(\widehat y) + \langle x,f\rangle  \\
\mbox{s.t.}\quad &-x+y+g_1^k=z, \\
&-x + g_2^k=\widehat x, \\
&-y=\widehat y.
\end{aligned}
\end{eqnarray}
Next, the augmented Lagrangian function associated with the problem (\ref{con:equal_subproblem_dual}) is
\begin{eqnarray*}
L_{\sigma}(x,y,z,\widehat x,\widehat y;s,u,v)
& = & \widehat p_1^*(z) + p_2^*(\widehat x) + \widehat p_3^*(\widehat y) + \langle x,f\rangle + \frac{\sigma}{2}\|-x+y+g_1^k-z+\frac{1}{\sigma}s\|^2 \nonumber \\
&& + \frac{\sigma}{2}\|-x+g_2^k-\widehat x +\frac{1}{\sigma}u\|^2 + \frac{\sigma}{2}\|-y-\widehat y+\frac{1}{\sigma}v\|^2\nonumber\\
&&-\frac{1}{2\sigma}(\|s\|^{2}+\|u\|^{2}+\|v\|^{2}).
\end{eqnarray*}
Then we describe the details of how to solve the subproblems.

\begin{itemize}
\item[(1)] \textbf{$(x, y)$-subproblem:} Let $(\bar x,\bar y)$ denote the optimal solution of the $(x,y)$-subproblem, i.e.,
\begin{eqnarray*}
(\bar{x},\bar{y}) &= \argmin_{x,y} \Big\{ L_{\sigma}(x,y,z,\widehat x,\widehat y;s,u,v)\Big\}.
\end{eqnarray*}
It is equivalent to solving the following linear system of equations
\begin{eqnarray*}
\begin{bmatrix}
2I &-I \\
-I &2I
\end{bmatrix}
\begin{bmatrix}
x\\
y
\end{bmatrix}& =
\frac{1}{\sigma}
\begin{bmatrix}
s + u -f -\sigma (z + \widehat x - g^{k}_{1} - g^{k}_{2})\\
v - s + \sigma (z - \widehat y - g^{k}_{1})
\end{bmatrix}.
\end{eqnarray*}

\item[(2)] \textbf{$(z,\widehat x, \widehat y)$-subproblem:} Let $(\bar{z},\bar{\widehat{x}},\bar{\widehat{y}})$ be the optimal solution of the
$(z,\widehat x, \widehat y)$-subproblem. We have
\begin{eqnarray*}
\bar z &=& \argmin_{z}\;\Big\{ \widehat p_1^*(z) + \frac{\sigma}{2}\| -x+y+g_1^k-z{\color{blue}+}\frac{1}{\sigma}s \|^2\Big\} \\
&=&{\rm Prox}_{\frac{1}{\sigma}\widehat p_1^*} (\frac{s}{\sigma}-x +y +g_1^k), \\
\bar{\widehat{x}} &=& \argmin_{\widehat x} \;\Big\{ p_2^*(\widehat x) + \frac{\sigma}{2}\| -x+g_2^k-\widehat x+\frac{1}{\sigma}u \|^2 \Big\}\\
&=&{\rm Prox}_{\frac{1}{\sigma}p_2^*} (\frac{u}{\sigma}-x+g_2^k), \\
\bar{\widehat{y}} &=& \argmin_{\widehat y} \;\Big\{ \widehat p_3^*(\widehat y)  + \frac{\sigma}{2}\| -y-\widehat{y}+\frac{1}{\sigma}v \|^2 \Big\}\\
&=&{\rm Prox}_{\frac{1}{\sigma}\widehat p_3^*} (\frac{v}{\sigma}-y).
\end{eqnarray*}
\end{itemize}
Based on the Moreau identity, we have
\begin{eqnarray*}
{\rm Prox}_{\frac{1}{\sigma}\widehat p_1^*} (\frac{s}{\sigma}-x +y +g_1^k)&=& \frac{1}{\sigma}(s - \sigma x + \sigma y + \sigma g_1^k- {\rm Prox}_{\sigma \widehat p_1}(s - \sigma x + \sigma y + \sigma g_1^k)),\\
{\rm Prox}_{\frac{1}{\sigma}p_2^*} (\frac{u}{\sigma}-x+g_2^k)&=& \frac{1}{\sigma}(u - \sigma x + \sigma g_2^k-{\rm Prox}_{\sigma p_2}(u - \sigma x + \sigma g_2^k)), \\
{\rm Prox}_{\frac{1}{\sigma}\widehat p_3^*} (\frac{v}{\sigma}-y)&=& \frac{1}{\sigma}(v - \sigma y -{\rm Prox}_{\sigma \widehat p_3}(v- \sigma y)).
\end{eqnarray*}
Note that
\begin{eqnarray}
{\rm Prox}_{\sigma \widehat p_1}(s) &=& \argmin_{x} \Big\{ \| \nabla_y x \|_1
+ \frac{1}{2 \widetilde \sigma_{k}} \| x - s^{k} \|^2 + \frac{1}{2 \sigma} \| x - s \|^2\Big\}\nonumber \\
&=& \argmin_{x} \Big\{ \| \nabla_y x \|_1 + \frac{1}{2 \widehat \sigma} \| x - \widehat s \|^2\Big\}\nonumber \\
&=& {\rm Prox}_{\widehat \sigma p_1}(\widehat s),\label{eq:prox-p1}
\end{eqnarray}
where $\widehat \sigma = \sigma \widetilde\sigma_{k} /(\sigma + \widetilde\sigma_{k}) < \sigma$, $\widehat s = \widehat \sigma (s / \sigma + s^{k} / \widetilde\sigma_{k})$.

As we know, for $s_{v} \in {\mathbb R}^n,\,\lambda > 0$, let $\| \nabla \cdot\|_1$ denote the first order difference operator of a vector, then the proximal mapping of $\| \nabla \cdot\|_1$
\begin{eqnarray}
{\rm Prox_{\lambda \,\| \nabla\cdot\|_1}}(s_{v})
= \arg \min_{x\in\mathbb{R}^{n}}\Big\{\| \nabla x_{v}\|_1 + \frac{1}{2\lambda} \| x_{v}-s_{v}\|^2\Big\}\label{eq:prox-nable}
\end{eqnarray}
can be obtained by the fast algorithm proposed by Condat \cite{Condat}.
 Hence, ${\rm Prox}_{\widehat \sigma p_1}(\widehat s)$ in \eqref{eq:prox-p1} can be computed column by column via \eqref{eq:prox-nable}, i.e.,
\begin{eqnarray}
{\rm Prox}_{\sigma \widehat p_3}(s)&=&\argmin_v \Big\{\widehat p_3 (v)+\frac{1}{2 \sigma} \| v-s \|^2\Big\}  \nonumber \\
&=&\argmin_v \Big\{\lambda_1 \| h(v) \|_1 - \langle g^{k}_h, h(v)\rangle +
\frac{1}{2 \sigma} \| v-s \|^2\Big\}, \label{argmin_v}
\end{eqnarray}
where $h(v)$ is a vector defined in Section \ref{sec:Introduction}.
Due to the separability of $\| \cdot \|_1$ and
$\| \cdot \|^2$, we can solve the problem (\ref{argmin_v}) in the following separate form
\begin{align}
v(:,i)=\arg \min_{v(:,i)} \Big\{ (\lambda_1-(g_h^k)_{i})\| v(:,i) \| + \frac{1}{2\sigma} \| v(:,i)-s(:,i) \|^2\Big\} ,\;i=1,\cdots,n,
\label{eq-v-problem-elementwise}
\end{align}
where $(g_h^k)_{i}$ is the $i$th element of the vector $g_h^{k}$. Since we take $\alpha=3.7$ in the SCAD function, we have $\lambda_1-(g_h^k)_{i}>0$, which guarantees that the problem \eqref{eq-v-problem-elementwise} is convex.

Now, we describe the dADMM for the problem
(\ref{con:equal_subproblem_dual}) as follows.

\begin{algorithm}
	\caption{The dADMM for the problem  (\ref{con:equal_subproblem_dual}):}\label{alg:dADMM}
	Let $\tau \in (0,(1+\sqrt{5})/2)$ be a scalar parameter, $\sigma >0$ be a given parameter, and choose $\{x^0,y^0,z^0,\widehat x^0,\widehat y^0,s^0,u^0,v^0\}$. For $j=0,1,2,\ldots$, iterate:
	\begin{description}
		\item[1.] Solve the following linear system to compute $(x^{j+1}, y^{j+1})$:
 \begin{eqnarray*}
 \begin{bmatrix}
 2I &-I \\
 -I &2I
 \end{bmatrix}
 \begin{bmatrix}
 x^{j+1}\\
 y^{j+1}
 \end{bmatrix}&=&
 \frac{1}{\sigma}
 \begin{bmatrix}
 s^j + u^j -f -\sigma (z^j + \widehat x^j - g_1^k - g_2^k)\\
 v^j - s^j + \sigma (z^j - \widehat y^j - g_1^k)
 \end{bmatrix}.
 \end{eqnarray*}
		\item[2.] Compute $(\widehat x^{j+1}, \widehat y^{j+1},z^{j+1})$:
 \begin{eqnarray*}
 \begin{bmatrix}
 \widehat x^{j+1}\\
 \widehat y^{j+1}\\
 z^{j+1}
 \end{bmatrix}&=&
 \begin{bmatrix}
 {\rm Prox}_{\frac{1}{\sigma} p_2^*} (\frac{1}{\sigma} u^j - x^{j+1} + g_2^k) \\
 {\rm Prox}_{\frac{1}{\sigma} \widehat p_3^*} (\frac{1}{\sigma} v^j - y^{j+1} ) \\
 {\rm Prox}_{\frac{1}{\sigma} \widehat p_1^*} (\frac{1}{\sigma} s^j - x^{j+1} +y^{j+1} + g_1^k)
 \end{bmatrix}.
 \end{eqnarray*}
 \item[3.] Update $(s^{j+1}, u^{j+1}, v^{j+1})$:
\begin{eqnarray*}
 s^{j+1}&=& s^j + \tau \sigma (-x^{j+1} + y^{j+1} + g_1^k - z^{j+1}), \\
 u^{j+1}&=& u^j + \tau \sigma (-x^{j+1} + g_2^k - \widehat x^{j+1}), \\
 v^{k+1}&=& v^j + \tau \sigma (-y^{j+1} - \widehat y^{j+1}).
\end{eqnarray*}
 \item[4.] If a stopping criterion is not satisfied, go to \textbf{Step 1}.
	\end{description}
\end{algorithm}

\section{Convergence Analysis}
\label{sec:Convergence analysis}
In this section, we analyze the convergence of the algorithm.
\subsection{Convergence analysis for the inexact dADMM}
\label{subsec:Convergence analysis for ADMM}
Since the inexact dADMM for the inner subproblem can be regarded as a special case of the inexact symmetric Gauss-Seidel method \cite{ChenST}, the convergence analysis can also be borrowed to here.
\begin{theorem}
Suppose that the solution set to the KKT system \eqref{eq:KKT-subproblem} is nonempty and the sequence $\{(s^j,u^j,v^j,x^j,y^j,z^j,\widehat x^j,\widehat y^j)\}$ is generated by the dADMM, then $\{(s^j,u^j,v^j)\}$ converges to the solution of the primal problem \eqref{con:subproblem} and $\{(x^j,y^j,z^j,\widehat x^j,\widehat y^j)\}$ converges to the solution of the dual problem \eqref{con:equal_subproblem_dual}.
\end{theorem}
\subsection{Convergence analysis for the PMM algorithm}
\label{subsec:Convergence analysis for PMM}
In this subsection, we focus on the convergence analysis of the PMM algorithm. Since the inner subproblem is solved inexactly, the error must be considered if analyzing the convergence of the PMM algorithm. When solving the subproblem \eqref{con:equal_subproblem_dual}, we actually solve the following problem
\begin{eqnarray}
\begin{aligned}
\label{con:inexact_subproblem_dual}
\min_{x,y,z,\widehat x,\widehat y } &\Big\{\widehat p_1^*(z+r_{z}^{k+1}) + p_2^*(\widehat x+r_{\widehat{x}}^{k+1}) + \widehat p_3^*(\widehat y+r_{\widehat{y}}^{k+1}) + \langle x,f\rangle +\\
 &\langle x,r_{x}^{k+1}\rangle + \langle y,r_{y}^{k+1}\rangle + \langle z,r_{z}^{k+1}\rangle + \langle\widehat{x},r_{\widehat{x}}^{k+1}\rangle + \langle \widehat{y},r_{\widehat{y}}^{k+1}\rangle \Big\} \\
\mbox{s.t.}\quad &-x+y+g_1^k-z-r_{1}^{k+1}=0, \\
&-x + g_2^k-\widehat x-r_{2}^{k+1}=0, \\
&-y-\widehat y-r_{3}^{k+1}=0,
\end{aligned}
\end{eqnarray}
where $r_{1}^{k+1},r_{2}^{k+1},r_{3}^{k+1},r_{x}^{k+1},r_{y}^{k+1},r_{z}^{k+1},r_{\widehat{x}}^{k+1},r_{\widehat{y}}^{k+1}\in \mathbb{R}^{m\times n}$ are error matrices. Therefore, the original primal problem we actually solve is
\begin{eqnarray}
\begin{aligned}
\label{con:inexact_subproblem_primal}
\min_{s} &\Big\{\widehat{p}_1(s-r_{z}^{k+1}) + p_2(f-s+r_{x}-r_{\widehat{x}}^{k+1}) + \widehat p_3(s+r_{y}^{k+1}-r_{\widehat{y}}^{k+1}) - \\
 &\langle s,g^{k}_{1}-r_{1}^{k+1}\rangle -\langle f-s+r_{x}^{k+1},g^{k}_{2}-r_{2}^{k+1}\rangle - \langle s+r_{y}^{k+1},-r_{3}^{k+1}\rangle-\\
 &\inprod{r_{z}^{k+1}}{s-r_{z}^{k+1}}-\inprod{r_{\widehat{x}}^{k+1}}{f-s+r_{x}^{k+1}-r_{\widehat{x}}^{k+1}} - \inprod{r_{\widehat{y}}^{k+1}}{s+r_{y}^{k+1}-r_{\widehat{y}}^{k+1}} \Big\}. \\
\end{aligned}
\end{eqnarray}
Let $r^{k+1}:=(r_{1}^{k+1},r_{2}^{k+1},r_{3}^{k+1},r_{x}^{k+1},r_{y}^{k+1},r_{z}^{k+1},r_{\widehat{x}}^{k+1},r_{\widehat{y}}^{k+1})\in \mathbb{R}^{m\times n}$. Note that $r^{k+1}$ converges to zero, it follows that if the inner subproblem is adequately iterated, then $r^{k+1}$ satisfies
\begin{eqnarray}
2p_{1}(r_{z}^{k+1})+2p_{2}(r_{x}^{k+1}-r_{\widehat{x}}^{k+1})+2p_{3}(r_{y}^{k+1}-r_{\widehat{y}}^{k+1})\nn\\
+\frac{1}{2\widetilde{\sigma}_{k}}\|r_{z}^{k+1}\|^{2} +\textrm{pert}(r^{k+1})\leq\frac{1}{4\widetilde{\sigma}_{k}}\|s^{k+1}-s^{k}\|^{2},\label{ineq:stopping criterion}
\end{eqnarray}
where
\begin{eqnarray*}
\textrm{pert}(r^{k+1}):=&|\inprod{s}{r_{1}^{k+1}}+\inprod{f-s+r_{x}^{k+1}}{r_{2}^{k+1}}+\inprod{s+r_{y}^{k+1}}{r_{3}^{k+1}}+\inprod{r_{z}^{k+1}}{s-r_{z}^{k+1}}\\
&+\inprod{r_{\widehat{x}}^{k+1}}{f-s+r_{x}^{k+1}-r_{\widehat{x}}^{k+1}}+\inprod{r_{\widehat{y}}^{k+1}}{s+r_{y}^{k+1}-r_{\widehat{y}}^{k+1}}|.
\end{eqnarray*}

In the following, for the convenience of statement, we denote
\begin{eqnarray*}
f_{k}(s) &:=& \widehat p_1(s-r_{z}^{k+1}) + p_2(f-s+r_{x}^{k+1}-r_{\widehat{x}}^{k+1}) + \widehat p_3(s+r_{y}^{k+1}-r_{\widehat{y}}^{k+1})-q_{1}(\nabla_{y}s^{k})\\
&& -\inprod{g_{1}^{k}}{s-s^{k}}-q_{2}(\nabla_{x}f-\nabla_{x}s^{k})-\inprod{g_{2}^{k}}{(f-s)-(f-s^{k})}+f^P_{k}(s).
\end{eqnarray*}
where $f^P_{k}(s):=\inprod{s}{r_{1}^{k+1}}+\inprod{f-s+r_{x}^{k+1}}{r_{2}^{k+1}}+\inprod{s+r_{y}^{k+1}}{r_{3}^{k+1}}+\inprod{r_{z}^{k+1}}{s-r_{z}^{k+1}}+\inprod{r_{\widehat{x}}^{k+1}}{f-s+r_{x}^{k+1}-r_{\widehat{x}}^{k+1}}+\inprod{r_{\widehat{y}}^{k+1}}{s+r_{y}^{k+1}-r_{\widehat{y}}^{k+1}}$.
Then at the $k$th iteration of the PMM algorithm, the actual problem we solve is
\begin{eqnarray*}
s^{k+1} = \argmin_{s}\Big\{f_{k}(s)\Big\}.
\end{eqnarray*}

For the proof of the convergence theory, we must first prove the sufficiently decent property of the sequence $\{g(s^{k})\}$.
\begin{lemma}\label{lemma-g-descent}
Let $\{s^{k}\}$ be a sequence generated by the PMM algorithm. We have the following descent property
\begin{eqnarray*}
g(s^{k})\geq g(s^{k+1})+\frac{1}{4\widetilde{\sigma}}_{k}\|s^{k+1}-s^{k}\|^{2}.
\end{eqnarray*}
\end{lemma}
\begin{proof}
Firstly, one can easily prove that the function $p_{i}\,(i=1,2\textrm{ or }3)$ satisfies
\begin{eqnarray}
|p_{i}(s_{1})-p_{i}(s_{2})| \leq p_{i}(s_{1}-s_{2}),\ \forall\, s_{1},s_{2}\in\mathbb{R}^{m\times n}.
\label{ineq:triangle inequality}
\end{eqnarray}
Then, by \eqref{ineq:triangle inequality} we have
\begin{eqnarray}
f_{k}(s^{k}) &=& g(s^{k})+p_{1}(s^{k}-r_{z}^{k+1})-p_{1}(s^{k})+p_{2}(f-s^{k}+r_{x}^{k+1}-r_{\widehat{x}}^{k+1})-p_{2}(f-s^{k})\nn\\
&&+p_{3}(s^{k}+r_{y}^{k+1}-r_{\widehat{y}}^{k+1})-p_{3}(s^{k})+\frac{1}{2\widetilde{\sigma}}_{k}\|r_{z}^{k+1}\|^{2}+ f^P_{k}(s^{k}),\nn\\
&\leq& g(s^{k})+p_{1}(r_{z}^{k+1})+p_{2}(r_{x}^{k+1}-r_{\widehat{x}}^{k+1})+p_{3}(r_{y}^{k+1}-r_{\widehat{y}}^{k+1})+\frac{1}{2\widetilde{\sigma}}_{k}\|r_{z}^{k+1}\|^{2}\nn\\
&&+ \textrm{pert}(r^{k+1}).\label{ineq:fk-upper-estimate}
\end{eqnarray}
In addition, since $q_{i}\,(i=1,2,3)$ is a convex function and $s^{k+1}$ is a minimizer of the function $f_{k}$, we have
\begin{eqnarray}
f_{k}(s^{k}) &\geq& f_{k}(s^{k+1})\nn\\
&\geq& p_{1}(s^{k}-r_{z}^{k+1})+\frac{1}{2\widetilde{\sigma}_{k}}\|s^{k+1}-s^{k}\|^{2}-q_{1}(\nabla_{y}s^{k+1})+\nn\\
&& p_{2}(f-s^{k+1}+r_{x}^{k+1}-r_{\widehat{x}}^{k+1})-q_{2}(\nabla_{x}f-\nabla_{x}s^{k+1})+p_{3}(s^{k+1}+r_{y}^{k+1}-r_{\widehat{y}}^{k+1})\nn\\
&& -q_{3}(h(s^{k+1})),\nn\\
&\geq& g(s^{k+1})-p_{1}(r_{z}^{k+1})-p_{2}(r_{x}^{k+1}-r_{\widehat{x}}^{k+1})-p_{3}(r_{y}^{k+1}-r_{\widehat{y}}^{k+1})+\frac{1}{2\widetilde{\sigma}_{k}}\|s^{k+1}-s^{k}\|^{2}.\nn\\
&&\label{ineq:fk-lower-estimate}
\end{eqnarray}
The last inequality in \eqref{ineq:fk-lower-estimate} is according to \eqref{ineq:triangle inequality}. In combination of \eqref{ineq:fk-upper-estimate} and \eqref{ineq:fk-lower-estimate}, we obtain
\begin{eqnarray}
g(s^{k})&\geq& g(s^{k+1})-2p_{1}(r_{z}^{k+1})-2p_{2}(r_{x}^{k+1}-r_{\widehat{x}}^{k+1})-2p_{3}(r_{y}^{k+1}-r_{\widehat{y}}^{k+1})-\frac{1}{2\widetilde{\sigma}}_{k}\|r_{z}^{k+1}\|^{2}\nn\\
&&- \textrm{pert}(r^{k+1})+\frac{1}{2\widetilde{\sigma}_{k}}\|s^{k+1}-s^{k}\|^{2}.\label{ineq:g-lower-estimate}
\end{eqnarray}
According to the condition \eqref{ineq:stopping criterion}, we get the result of this lemma.
\end{proof}

In the following, we present the result about the criterion of being a d-stationary point. But before that, we give a definition of a function. For $\widetilde{s}\in\mathbb{R}^{m\times n}$, $\widetilde{\sigma}>0$,
\begin{eqnarray*}
\widehat{f}(s;\widetilde{s},\widetilde{\sigma}) &:=& p_1(s) + p_2(f-s) + p_3(s)-q_{1}(\nabla_{y}\widetilde{s})-\inprod{\nabla_{y}^{*}\nabla q_{\alpha,\lambda_{1}}(\nabla_{y}\widetilde{s})}{s-\widetilde{s}}-\\
&&q_{\alpha,\lambda_{2}}(\nabla_{x}(f-\widetilde{s}))-\inprod{\nabla_{x}^{*}\nabla q_{\alpha,\lambda_{2}}(\nabla_{x}f-\nabla_{x}\widetilde{s})}{s-\widetilde{s}}\\
&& -q_{\alpha,\lambda_{3}}(h(\widetilde{s}))-\inprod{\nabla q_{\alpha,\lambda_{3}}(h(\widetilde{s}))}{h(s)-h(\widetilde{s})}+\frac{1}{2\widetilde{\sigma}}\|s-\widetilde{s}\|^{2}.
\end{eqnarray*}
Then in order to prove the result, we need a lemma. Since it is very similar to \cite[Lemma 5, Lemma6]{CuiPS}, we present it without proof.
\begin{lemma}\label{lemma-equivalence}
Let $\bar{s}\in \mathbb{R}^{m\times n}$, then $\bar{s}$ is a d-stationary point of \eqref{eq:g-function} if and only if there exist $\widetilde{\sigma}\geq 0$ such that $\bar{s}$ is a solution of the following minimization problem
\begin{eqnarray*}
\min_{s\in\mathbb{R}^{m\times n}}\Big\{\widehat{f}(s;\widetilde{s},\widetilde{\sigma})\Big\}.
\end{eqnarray*}
\end{lemma}

Now we present the convergence result of the PMM algorithm.
\begin{theorem}\label{thm-convergence}
Assume that the objective function in \eqref{eq:g-function} is bounded below and $\{\widetilde{\sigma}_{k}\}$ are positive convergent sequences. Then every cluster point of the sequence $\{s^{k}\}$ generated by the PMM algorithm is a d-stationary point of \eqref{eq:g-function}.
\end{theorem}
\begin{proof}
It is known from Lemma \ref{lemma-g-descent} that $\{g(s^{k})\}$ is a non-increasing sequence. By the assumption that the value of $g(s)$ is bounded below, we obtain that $\{g(s^{k})\}$ converges. Meanwhile, $\lim_{k\rightarrow\infty}\|s^{k+1}-s^{k}\|=0$. Let $\{s^{k}\}_{k\in\mathcal{K}}$ be a subsequence of $\{s^{k}\}$, and $\lim_{k(\in\mathcal{K})\rightarrow\infty}s^{k}=s^{\infty}$. In the following, we prove that $s^{\infty}$ is a d-stationary point of \eqref{eq:g-function}.

Since $s^{k+1}$ is a minimizer of the function $f_{k}$, we have
\begin{eqnarray*}
&& \widehat{f}(s;s^{k},\widetilde{\sigma}_{k})+p_{1}(r_{z}^{k+1})+p_{2}(r_{x}^{k+1}-r_{\widehat{x}}^{k+1})+p_{3}(r_{y}^{k+1}-r_{\widehat{y}}^{k+1})+\frac{1}{2\widetilde{\sigma}_{k}}\|r_{z}^{k+1}\|^{2}\\
&& -\inprod{g_{1}^{k}}{r_{z}^{k+1}}+\inprod{g_{2}^{k}}{r_{x}^{k+1}-r_{\widehat{x}}^{k+1}}-\inprod{|\nabla q_{3}(h(s^{k}))|}{h(r_{y}^{k+1}-r_{\widehat{y}}^{k+1})}\\
&\geq& f_{k}(s)\geq f_{k}(s^{k+1})\\
&\geq& \widehat{f}(s^{k+1};s^{k},\widetilde{\sigma}_{k})-p_{1}(r_{z}^{k+1})-p_{2}(r_{x}^{k+1}-r_{\widehat{x}}^{k+1})-p_{3}(r_{y}^{k+1}-r_{\widehat{y}}^{k+1})-\frac{1}{2\widetilde{\sigma}_{k}}\|r_{z}^{k+1}\|^{2}\\
&& +\inprod{g_{1}^{k}}{r_{z}^{k+1}}-\inprod{g_{2}^{k}}{r_{x}^{k+1}-r_{\widehat{x}}^{k+1}}+\inprod{|\nabla q_{3}(h(s^{k}))|}{h(r_{y}^{k+1}-r_{\widehat{y}}^{k+1})},\ \forall\,s\in \mathbb{R}^{m\times n}.
\end{eqnarray*}
We assume $\lim_{k\rightarrow\infty}\widetilde{\sigma}_{k}=\widetilde{\sigma}_{\infty}$. Note that $\|r^{k+1}\|\rightarrow 0$ due to \eqref{ineq:stopping criterion} and $h$ is a continuous function. Then letting $k(\in\mathcal{K})\rightarrow\infty$, we derive
\begin{eqnarray*}
\widehat{f}(s;s^{\infty},\widetilde{\sigma}_{\infty})
\geq \widehat{f}(s^{\infty};s^{\infty},\widetilde{\sigma}_{\infty}),\ \forall\, s^{\infty}\in\mathbb{R}^{m\times n}.
\end{eqnarray*}

Hence, by Lemma \ref{lemma-equivalence} we can get the desired result of this theorem.
\end{proof}

Finally, we establish the local convergence rate of the sequence $\{s^{k}\}$ in the following. But before that, we present a lemma.
\begin{lemma}
\label{lemma-2norm-upper-Lipschtz}
In $\mathbb{R}^{m}$, define a function $w:\,\mathbb{R}^{m}\rightarrow \mathbb{R}$, $w(x_{v}):=\|x_{v}\|$, then there is a constant $\kappa\geq 0$ such that
\begin{eqnarray}\label{ineq:upper Lipschitz}
\partial w(x_{v})\subseteq \partial w(\bar{x}_{v})+\kappa\|x_{v}-\bar{x}_{v}\| \mathbb{B}_{m},\ \ \forall\,x_{v}\in\mathbb{R}^{m}.
\end{eqnarray}
\begin{proof}
We discuss this problem in three cases: \\
(i). If $x_{v}\neq 0$ and $\bar{x}_{v}\neq0$, then $\partial w(x_{v})=\Big\{\frac{x_{v}}{\|x_{v}\|}\Big\}$, $\partial w(\bar{x}_{v})=\Big\{\frac{\bar{x}_{v}}{\|\bar{x}_{v}\|}\Big\}$, we can see that the Lipschitz continuity holds, which implies that the result \eqref{ineq:upper Lipschitz} holds.\\
(ii). If $x_{v}\neq 0$ and $\bar{x}_{v}=0$, then $\partial w(x_{v})=\Big\{\frac{x_{v}}{\|x_{v}\|}\Big\}$, $\partial w(\bar{x}_{v})=\mathbb{B}_{m}$, the result \eqref{ineq:upper Lipschitz} obviously holds.\\
(iii). If $x_{v}=0$ and $\bar{x}_{v}\neq0$, then $\partial w(x_{v})=\mathbb{B}_{m}$, $\partial w(\bar{x}_{v})=\Big\{\frac{\bar{x}_{v}}{\|\bar{x}_{v}\|}\Big\}$, therefore \eqref{ineq:upper Lipschitz} holds with $\kappa\geq\frac{1}{\|x_{v}-\bar{x}_{v}\|}$.
\end{proof}
\end{lemma}

\begin{theorem}
Assume that the objective function in \eqref{eq:g-function} is bounded below. Let $\{s^{k}\}$ be the sequence generated by the PMM algorithm with $\mathcal{C}^{\infty}$ as the set of all its cluster points. The whole sequence $\{s^{k}\}$ converges to an element of $\mathcal{C}^{\infty}$, if one of the following conditions holds:\\
(I) The set $\mathcal{C}^{\infty}$ contains an isolated element;\\
(II) $\{s^{k}\}$ is a bounded sequence.\\
Furthermore, based on the condition (II), let $\lim\limits_{k\rightarrow\infty}s^{k}=s^{\infty}\in\mathcal{C}^{\infty}$ and the function $g$ satisfies the K{\L} property at $s^{\infty}$ with an exponent $\gamma\in[0,1)$, we have\\
(a) The sequence $\{s^{k}\}$ converges in a finite number of steps, if $\gamma = 0$;\\
(b) The sequence $\{s^{k}\}$ converges R-linearly, i.e., there exist $\mu>0$ and $\theta\in[0,1)$ such that $\|s^{k}-s^{\infty}\|\leq\mu\theta^{^{k}}$, if $0<\gamma\leq\frac{1}{2}$ and $k$ is sufficiently large;\\
(c) The sequence $\{s^{k}\}$ converges R-sublinearly, i.e., there exists $\mu>0$ such that $\|s^{k}-s^{\infty}\|\leq\mu k^{-\frac{1-\gamma}{2\gamma-1}}$, if $\frac{1}{2}<\gamma<1$ and $k$ is sufficiently large.
\end{theorem}
\begin{proof}
Since the objective function in \eqref{eq:g-function} is bounded below, it follows from the proof of Theorem \ref{thm-convergence} that $\lim\limits_{k\rightarrow\infty}\|s^{k+1}-s^{k}\|=0$. Under condition (I), the sequence converges to the isolated element of $\mathcal{C}^{\infty}$ due to \cite[Proposition 8.3.10]{Facchinei2003}. To prove the sequential convergence of $\{s^{k}\}$ under the condition (II), it suffices to prove the following three properties of $\{s^{k}\}$ and then apply the convergence results in \cite[Theorem 2.9]{AttouchBS} to the function $g$:\\
(i) The sequence $\{g(s^{k})\}$ is descent, i.e., $g(s^{k+1})\geq g(s^{k})+\frac{c}{2}\|s^{k+1}-s^{k}\|^{2},\,\forall\,k\geq 1$, where $c>0$ is a constant;\\
(ii) There exists a subsequence $\{s^{k}\}_{k\in\mathcal{K}}$ of $\{s^{k}\}$ such that $\lim\limits_{k(\in\mathcal{K})\rightarrow\infty}s^{k}=s^{\infty}$ and $\lim\limits_{k(\in\mathcal{K})\rightarrow\infty}g(s^{k})=g(s^{\infty})$;\\
(iii) There exist a constant $K>0$ and $\varepsilon^{k+1}\in\partial g(s^{k+1})$ such that $\|\varepsilon^{k+1}\|\leq K\|s^{k+1}-s^{k}\|$, for $k$ sufficiently large.\\
Obviously, the first two properties have been proven. Now we focus on the proof of the third property. Let
\begin{eqnarray*}
\bar{\varepsilon}^{k+1}:=g_{1}^{k}-g_{1}^{k+1}-g_{2}^{k}+g_{2}^{k+1}-\frac{1}{\widetilde{\sigma}_{k}}(s^{k+1}-s^{k}-r_{z}^{k+1})-r_{1}^{k+1}-r_{z}^{k+1}.
\end{eqnarray*}
Then $\bar{\varepsilon}^{k+1}\in\partial g_{k}(s^{k+1})$, where
\begin{eqnarray*}
g_{k}(s)&=&p_{1}(s-r_{z}^{k+1})+p_{2}(f-s+r_{x}^{k+1}-r_{\hat{x}}^{k+1})+p_{3}(s+r_{y}^{k+1}-r_{\widehat{y}}^{k+1})\\
&&-q_{\alpha,\lambda_{1}}(\nabla_{y}s)-q_{\alpha,\lambda_{2}}(\nabla_{x}f-\nabla_{x}s)-q_{\alpha,\lambda_{3}}(h(s+r_{y}^{k+1}-r_{\widehat{y}}^{k+1})).
\end{eqnarray*}
According to the special structure of the function $h$, we may consider every element $h_{i}$ separately. Since the functions $\partial p_{1}$ and $\partial p_{2}$ are piecewise polyhedral, $\|r_{z}^{k+1}\|$ and $\|r_{x}^{k+1}-r_{\widehat{x}}^{k+1}\|$ are sufficiently small for all $k$ sufficiently large, we can deduce from \cite{Robinson1981} that $\partial p_{1}$ and $\partial p_{2}$ are locally upper Lipschitz continuous. Hence, it is known from Lemma \ref{lemma-2norm-upper-Lipschtz} and \cite[Theorem 23.8]{Rockafellar70} that there exist $\kappa_{1},\kappa_{2},\kappa_{3}\geq0$, such that
\begin{eqnarray*}
\bar{\varepsilon}^{k+1}&\in&\partial p_{1}(s^{k+1}-r_{z}^{k+1})+\partial p_{2}(f-s^{k+1}+r_{x}^{k+1}-r_{\widehat{x}}^{k+1})+\\
&&\partial p_{3}(s^{k+1}+r_{y}^{k+1}-r_{\widehat{y}}^{k+1})-\partial q_{\alpha,\lambda_{3}}(h(s^{k+1}+r_{y}^{k+1}-r_{\widehat{y}}^{k+1}))-g_{1}^{k+1}+g_{2}^{k+1}\\
&\subset& \partial p_{1}(s^{k+1})+\kappa_{1}\|r_{z}^{k+1}\|\mathbb{B}_{m\times n}+\partial p_{2}(f-s^{k+1})+\kappa_{2}\|r_{x}^{k+1}-r_{\hat{x}}^{k+1}\|\mathbb{B}_{m\times n}\\
&& +\partial p_{3}(s^{k+1})-\partial q_{\alpha,\lambda_{3}}(h(s^{k+1}))+\kappa_{3}\|r_{y}^{k+1}-r_{\widehat{y}}^{k+1}\|\mathbb{B}_{m\times n}-g_{1}^{k+1}+g_{2}^{k+1}\\
&=& \partial p_{1}(s^{k+1})+\partial p_{2}(f-s^{k+1})+\partial p_{3}(s^{k+1})-\partial q_{\alpha,\lambda_{3}}(h(s^{k+1}))+(\kappa_{1}\|r_{z}^{k+1}\|\\
&&+\kappa_{2}\|r_{x}^{k+1}-r_{\widehat{x}}^{k+1}\|+\kappa_{3}\|r_{y}^{k+1}-r_{\widehat{y}}^{k+1}\|)\mathbb{B}_{m\times n}-g_{1}^{k+1}+g_{2}^{k+1}.
\end{eqnarray*}
Hence, we can find $d\bar{\varepsilon}^{k+1}\in\mathbb{R}^{m\times n}$ with $\|d\bar{\varepsilon}^{k+1}\|\leq \kappa_{1}\|r_{z}^{k+1}\|+\kappa_{2}\|r_{x}^{k+1}-r_{\widehat{x}}^{k+1}\|+\kappa_{3}\|r_{y}^{k+1}-r_{\widehat{y}}^{k+1}\|$, such that $\varepsilon^{k+1}:=\bar{\varepsilon}^{k+1}+d\bar{\varepsilon}^{k+1}\in\partial g(s^{k+1})$.
From \eqref{ineq:stopping criterion}, we know that there exists a constant $\kappa_{4}>0$, such that $\|d\bar{\varepsilon}^{k+1}\|\leq\kappa_{4}\|s^{k+1}-s^{k}\|$. In combination of \eqref{ineq:stopping criterion} and the Lipschitz continuity of $\nabla_{y}^{*}\nabla q_{\alpha,\lambda_{1}}$ and $\nabla_{x}^{*}\nabla q_{\alpha,\lambda_{2}}$, we can find some $\kappa_{5}>0$, such that $\|\bar{\varepsilon}^{k+1}\|\leq\kappa_{5}\|s^{k+1}-s^{k}\|$. Therefore, we can get the result of (iii). Finally, since we have proved the properties (i)-(iii), the results for the convergence rate can be established similar to \cite[Theroem 2]{AttouchB} or \cite[Proposition 4]{bolte-pauwels2016}.
\end{proof}

\section{Numerical Experiments}
\label{sec:Numerical experiments}
In this section, we implement some numerical experiments to demonstrate the efficiency of our algorithm. All the experiments are conducted on a Desktop with Intel(R) Core(TM) i5-8250U, CPU@1.60GHz of 8G memory running 64 bit Windows operation system.
All the codes are written in {\sc Matlab} (R2018a) with some subroutines written in C.

In order to give an overall evaluation, we select the full-reference evaluation indices: the peak signal-to-noise ratio (PSNR) and the structural similarity (SSIM) \cite{wang2004image}, which are defined as follows.

\bdefi[Peak Signal-to-noise Ratio (PSNR)]
Given a clean image $u$ and a degraded image $v$ of $m \times n$, the PSNR is defined as
$$\textrm{PSNR}:=10 \log_{10}(\frac{u_{\max}^2}{\textrm{MSE}}),$$
where MSE denotes the mean square error and is defined as
$$\textrm{MSE}:=\frac{1}{mn}\sum_{i=1}^m \sum_{j=1}^n (u_{ij}-v_{ij})^2.$$
\edefi

\bdefi[Structual Similarity (SSIM) \cite{wang2004image}]
The mean values of the images $u$ and $v$ are denoted as
\begin{eqnarray*}
\mu_u &=\frac{1}{mn} \sum_{i=1}^m \sum_{j=1}^n u_{ij}
\end{eqnarray*}
and
\begin{eqnarray*}
\mu_v &=\frac{1}{mn} \sum_{i=1}^m \sum_{j=1}^n v_{ij},
\end{eqnarray*}
respectively. The standard deviations are denoted as
\begin{eqnarray*}
\sigma_u &= \Big( \frac{1}{mn-1} \sum_{i=1}^m \sum_{j=1}^n (u_{ij}-\mu_u)^2 \Big)^{\frac{1}{2}}
\end{eqnarray*}
and
\begin{eqnarray*}
\sigma_v &= \Big( \frac{1}{mn-1} \sum_{i=1}^m \sum_{j=1}^n (v_{ij}-\mu_v)^2 \Big)^{\frac{1}{2}},
\end{eqnarray*}
respectively. The covariance of the images $u$ and $v$ is denoted as
\begin{eqnarray*}
\sigma_{uv} &= \frac{1}{mn-1} \sum_{i=1}^m \sum_{j=1}^n (u_{ij}-\mu_u)(v_{ij}-\mu_v).
\end{eqnarray*}
And the factors of the luminance, contrast, and structure of the images $u$ and $v$ are defined as
$$l(u,v):=\frac{2 \mu_u \mu_v +c_1}{\mu_u^2 + \mu_v^2 +c_1},\;
c(u,v):=\frac{2\sigma_u \sigma_v+c_{2}}{\sigma_u^2 + \sigma_v^2 + c_2},\;
s(u,v):=\frac{\sigma_{uv}+c_3}{\sigma_u \sigma_v + c_3},$$
where $c_1=(k_1 L_u)^2,c_2=(k_2 L_u)^2,c_3=c_2/2$ are three constants. $k_1=0.01$ and $k_2=0.03$ are the common default values.
Then we can define
$$\textrm{SSIM}:=[l(u,v)]^{\alpha} [c(u,v)]^{\beta} [s(u,v)]^{\gamma}.$$

In order to simplify the expression, we set $\alpha,\beta,\gamma=1$ in this paper, then we have
$$\textrm{SSIM}=\frac{(2\mu_u \mu_v+c_1)(2\sigma_{uv}+c_2)}
{(\mu_u^2+\mu_v^2+c_1)(\sigma_u^2 + \sigma_v^2 + c_2)}.$$
\edefi

For the nonconvex model,
the iteration process of the outer problem is terminated if
$$\max\{R_p^{k},R_d^{k}\} < \texttt{Tol}=2 \times 10^{-4}$$
is satisfied, where
\begin{align*}
&R_p^{k} = \frac{\| f-s^{k}-u^{k} \| + \| s^{k}-v^{k} \|}{1+\| f \|},\\
&R_d^{k} = \frac{\| -x^{k}-\widehat x^{k} \| + \| -x^{k}+y^{k}-z^{k} \| + \| -y^{k}-\widehat y^{k} \|}{1+\| f \|},
\end{align*}
or the iteration number reaches $5$. The iteration process of the subproblem is terminated if
\begin{eqnarray*}
2p_{1}(r_{z}^{k+1})+2p_{2}(r_{x}^{k+1}-r_{\widehat{x}}^{k+1})+2p_{3}(r_{y}^{k+1}-r_{\widehat{y}}^{k+1})\nn\\
+\frac{1}{2\widetilde{\sigma}_{k}}\|r_{z}^{k+1}\|^{2} +\textrm{pert}(r^{k+1})\leq\frac{1}{4\widetilde{\sigma}_{k}}\|s^{k+1}-s^{k}\|^{2},
\end{eqnarray*}
or the iteration number reaches $100$.

For the convex model, the iteration process is terminated if
$$\max\{R_p^k,R_d^k,R_c^k\} < \texttt{Tol}=2 \times 10^{-4}$$
is satisfied, where
\begin{align*}
&R_p^k = \frac{\| f-s^{k}-u^{k} \| + \| s^{k}-v^{k} \|}{1+\| f \|},\\
&R_d^k = \frac{\| -x^{k}-\widehat x^{k} \| + \| -x^{k}+y^{k}-z^{k} \| + \| -y^{k}-\widehat y^{k} \|}{1+\| f \|},\\
&R_c^k = \frac{\| s^{k}-{\rm Prox}_{\widehat p_1}(s^{k}+z^{k}) \| + \| u^{k}-{\rm Prox}_{p_2}(u^{k}+\widehat x^{k}) \| + \| v-{\rm Prox}_{\widehat p_3}(v^{k}+\widehat y^{k}) \|}{1+\| f \|},
\end{align*}
or the iteration number reaches $500$.

Three real remote sensing stripe images with different stripe noise distributions are downloaded from the website https://ladsweb.nascom.nasa.gov/. Although there are three parameters in the model \eqref{target_function} and \eqref{target_fun}, we only need to do the cross validation for two parameters by dividing one parameter.

In Table 1, we present the experimental results of different algorithms on the convex model and the nonconvex model, respectively. The image is degraded by nonperiodical and periodical stripes, respectively. From the table, we can see that the nonconvex model performs better than the convex counterpart, especially for the PSNR. In addition, for the convex model, the dADMM performs better than the pADMM. This also demonstrates that we adopt the dADMM to solve the subproblems of the PMM is reasonable.

Furthermore, we also present the stripe and destripe images in Figures 1-3 to illustrate the effectiveness of the model. We can see that the proposed model can remove the stripes, and preserve the stripe-free information and the image details very well.
\begin{table}
\centering
\caption{Experimental results}
\begin{tabular}{*{6}{c|}c}
\hline
\multirow{2}*{Stripe noise} &\multirow{2}*{Case} &\multirow{2}*{Index} &\multirow{2}*{Degraded} &\multicolumn{2}{c|}{Convex model} &\multirow{2}*{Nonconvex model} \\
\cline{5-6}
&&&&pADMM &dADMM &\\
\hline
\multirow{15}*{Nonperiodical} &\multirow{5}*{Case 1}
&KKT    &&1.99e-4&1.99e-4 &1.99e-4 \\
&&PSNR  &23.05 &55.47&59.13 &63.36 \\
&&SSIM  &0.6486 &0.9995&0.9998 &0.9999 \\
&&time  &       & 26.89      &21.48 &39.24\\
&&iter &&478&233&434\\
\cline{2-7}
                      &\multirow{5}*{Case 2}
&KKT    &&1.99e-4&1.98e-4 &1.99e-4 \\
&&PSNR &18.27 &53.92&58.78 &62.43\\
&&SSIM &0.2753 &0.9995&0.9997 &0.9998\\
&&time &&24.38&17.56&33.21\\
&&iter &&401&218&382\\
\cline{2-7}
                      &\multirow{5}*{Case 3}
&KKT    &&1.99e-4&1.99e-4 &2.00e-4 \\
&&PSNR &24.33 &43.67&48.82 &55.38\\
&&SSIM &0.3626 &0.9943&0.9962 &0.9976 \\
&&time &&18.46&13.39&30.25\\
&&iter &&378&179&352\\
\hline
\multirow{15}*{Periodical} &\multirow{5}*{Case 1}
&KKT    &&1.99e-4&2.00e-4 &2.00e-4\\
&&PSNR &20.68 &52.00&57.13 &62.00\\
&&SSIM &0.5868 &0.9987&0.9995 &0.9999\\
&&time &&37.64&29.12&55.83\\
&&iter &&492&281&500\\
\cline{2-7}
                      &\multirow{5}*{Case 2}
&KKT    &&1.99e-4&1.99e-4 &1.79e-4 \\
&&PSNR &17.67 &40.75&46.69 &54.42\\
&&SSIM &0.4577 &0.9946&0.9960 &0.9992\\
&&time &&16.57&11.28&28.26\\
&&iter &&455&180&440\\
\cline{2-7}
                      &\multirow{5}*{Case 3}
&KKT    &&1.98e-4&1.99e-4 &1.96e-4 \\
&&PSNR &18.32 &41.05&48.36 &55.98\\
&&SSIM &0.3283&0.9949 &0.9970 &0.9995\\
&&time &&14.25&9.98&25.76\\
&&iter &&423&167&409\\
\hline
\end{tabular}
\end{table}

\newpage
\begin{figure}[!ht]
  \centering 
  \vspace{-0.35 cm} 
  \subfigtopskip=2pt 
  \subfigbottomskip=2pt 
  \subfigcapskip=-5pt 
  \graphicspath{{picture/}}
  \subfigure[]{
    \label{pa}
    \includegraphics[width=0.25\linewidth]{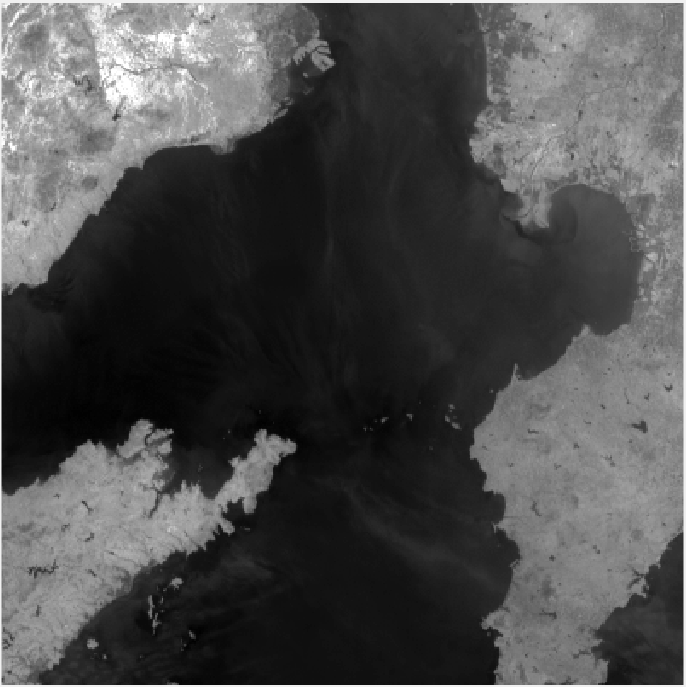}}
  \quad 
  \subfigure[]{
    \label{pb}
    \includegraphics[width=0.25\linewidth]{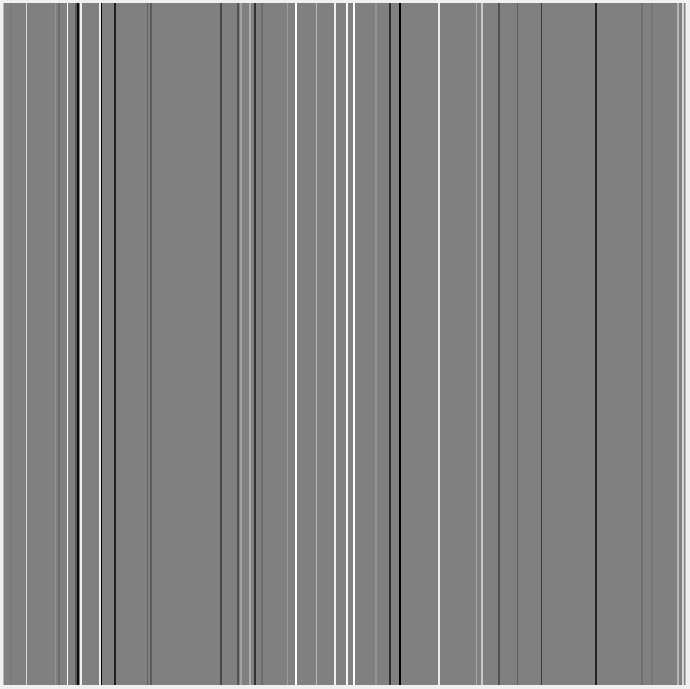}}
    \quad
  \subfigure[]{
	\label{pc}
	\includegraphics[width=0.25\linewidth]{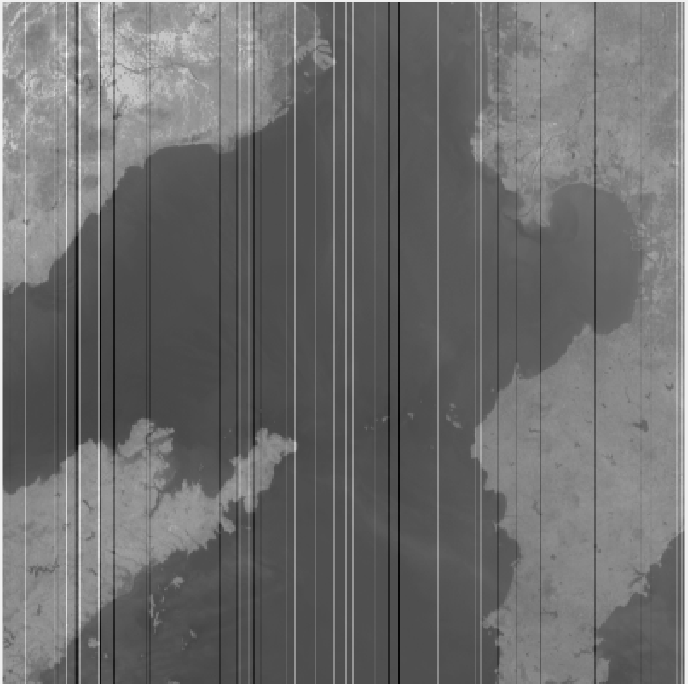}}\\
  \subfigure[]{
	\label{pd}
	\includegraphics[width=0.25\linewidth]{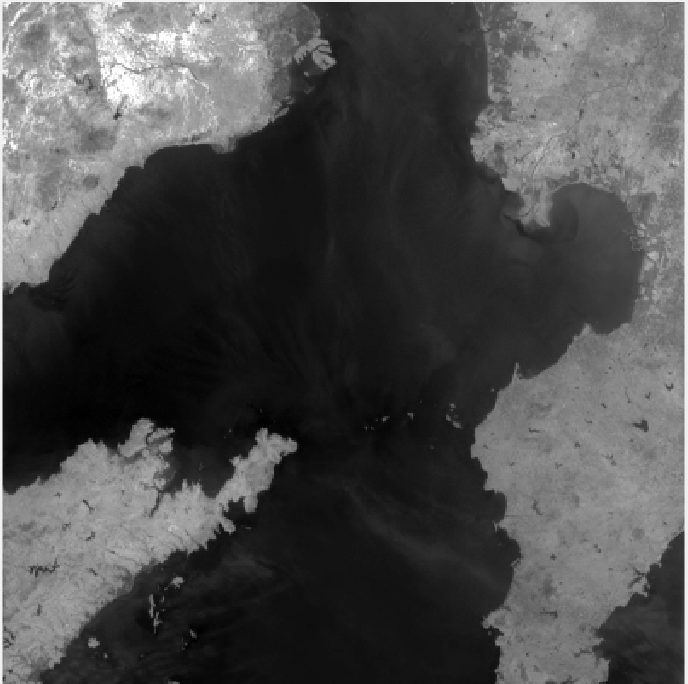}}
  \quad
  \subfigure[]{
    \label{pe}
    \includegraphics[width=0.25\linewidth]{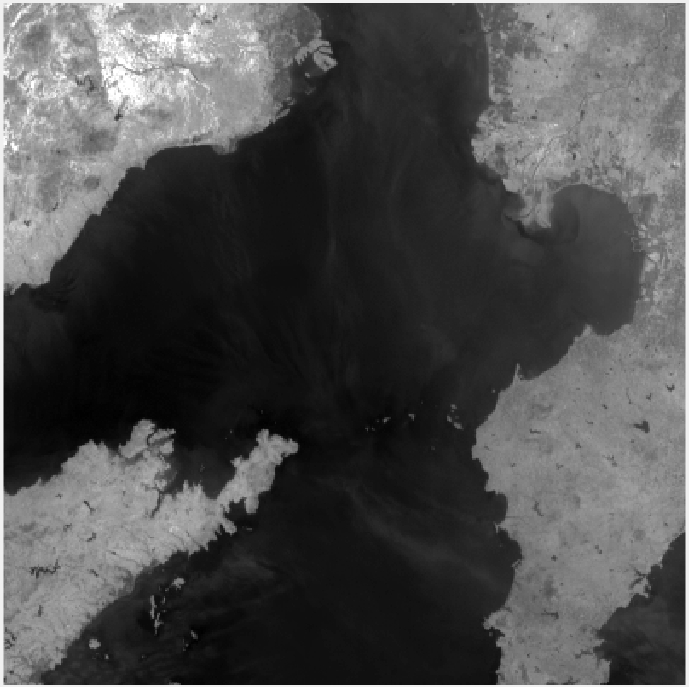}}
  \quad 
  \subfigure[]{
    \label{pf}
    \includegraphics[width=0.25\linewidth]{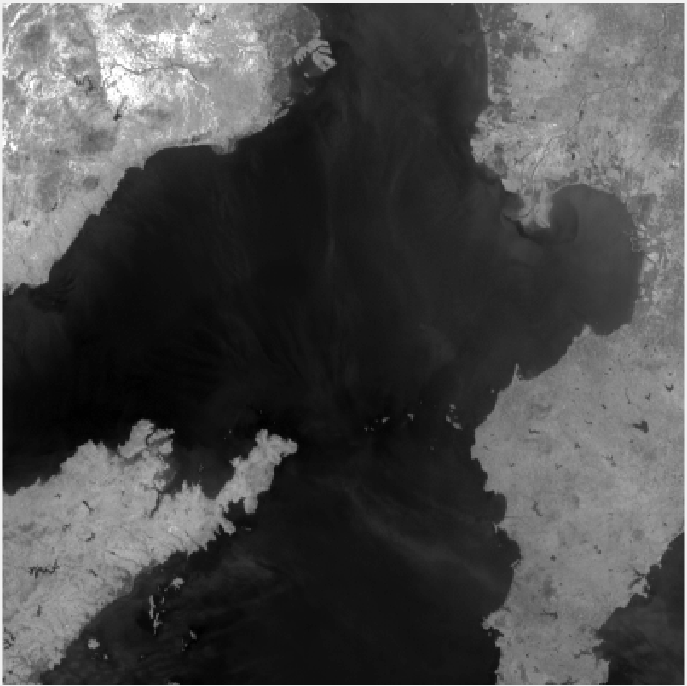}}
  \caption{The destripe results. (a) Original image; (b) stripe component; (c) degraded image; (d) estimated image by the convex model (pADMM); (e) estimated image by the convex model (dADMM); (f) estimated image by the nonconvex model}
  \label{pic_one}
\end{figure}

\newpage
\begin{figure}[!ht]
  \centering 
  \vspace{-0.35 cm} 
  \subfigtopskip=2pt 
  \subfigbottomskip=2pt 
  \subfigcapskip=-5pt 
  \graphicspath{{picture/}}
  \subfigure[]{
    \label{pa}
    \includegraphics[width=0.25\linewidth]{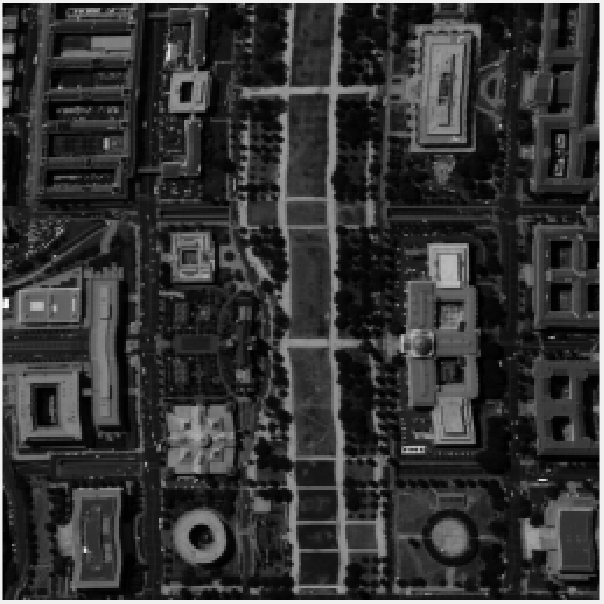}}
  \quad 
  \subfigure[]{
    \label{pb}
    \includegraphics[width=0.25\linewidth]{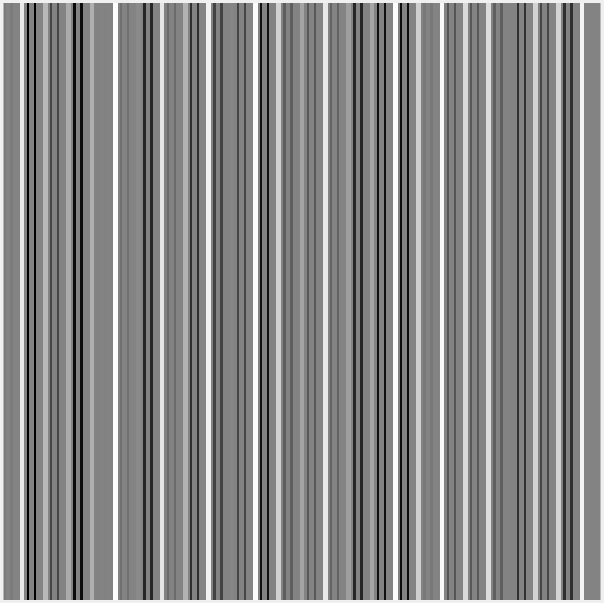}}
    \quad
  \subfigure[]{
	\label{pc}
	\includegraphics[width=0.25\linewidth]{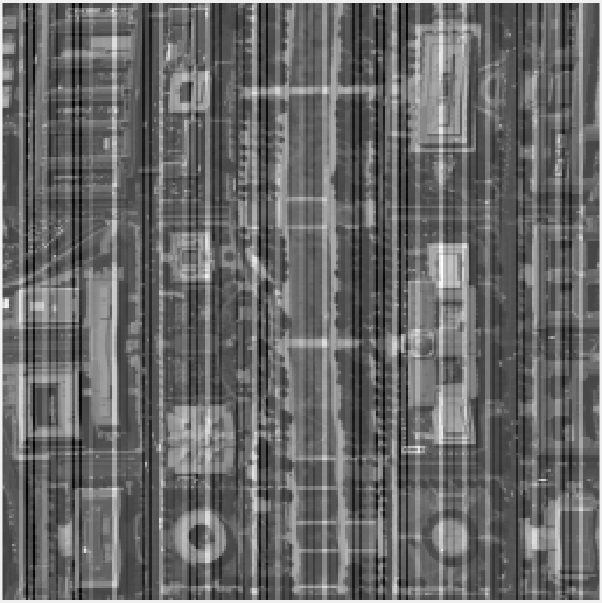}}\\
  \subfigure[]{
	\label{pd}
	\includegraphics[width=0.25\linewidth]{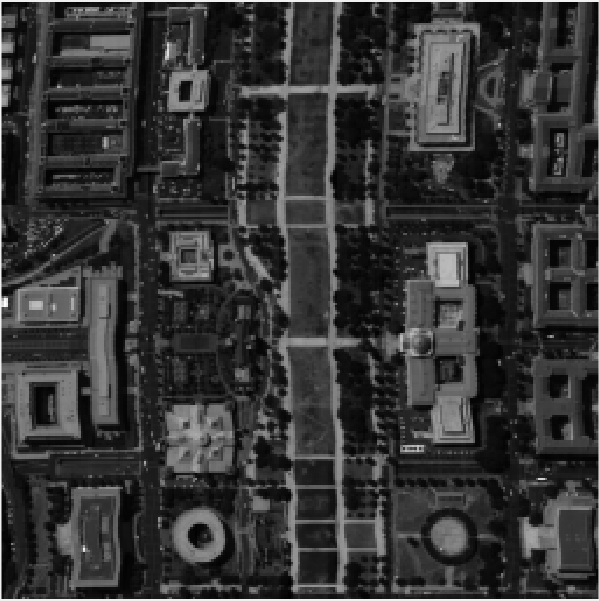}}
  \quad
  \subfigure[]{
    \label{pe}
    \includegraphics[width=0.25\linewidth]{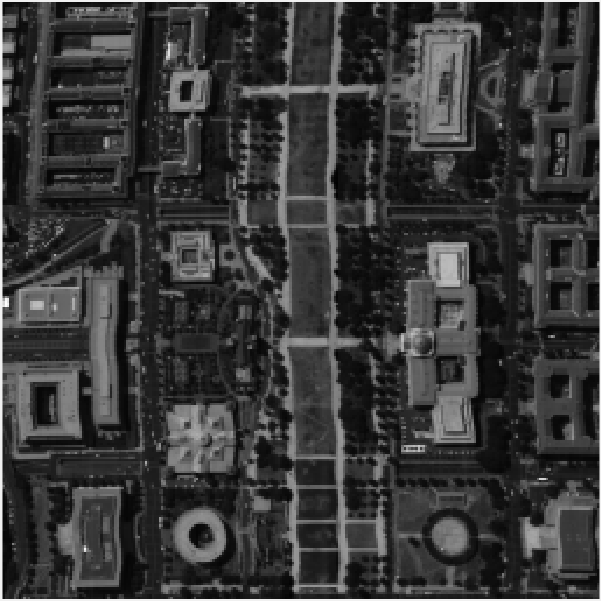}}
  \quad 
  \subfigure[]{
    \label{pf}
    \includegraphics[width=0.25\linewidth]{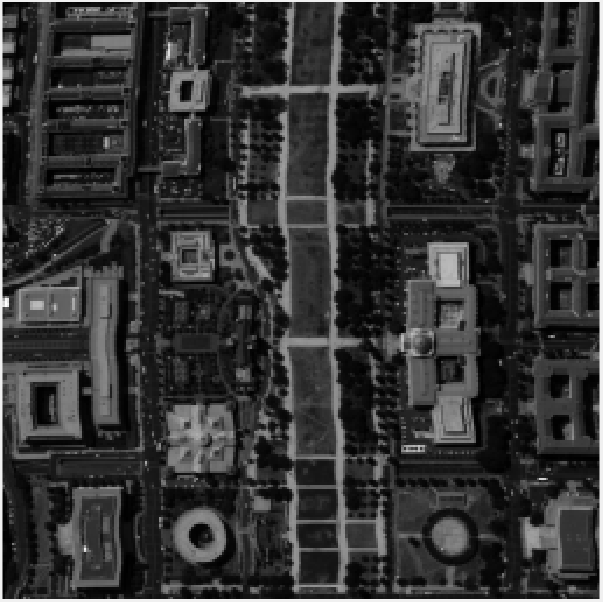}}
  \caption{The destripe results. (a) Original image; (b) stripe component; (c) degraded image; (d) estimated image by the convex model (pADMM); (e) estimated image by the convex model (dADMM); (f) estimated image by the nonconvex model}
  \label{pic_two}
\end{figure}

\begin{figure}[!ht]
  \centering 
  \vspace{-0.35 cm} 
  \subfigtopskip=2pt 
  \subfigbottomskip=2pt 
  \subfigcapskip=-5pt 
  \graphicspath{{picture/}}
  \subfigure[]{
    \label{pa}
    \includegraphics[width=0.25\linewidth]{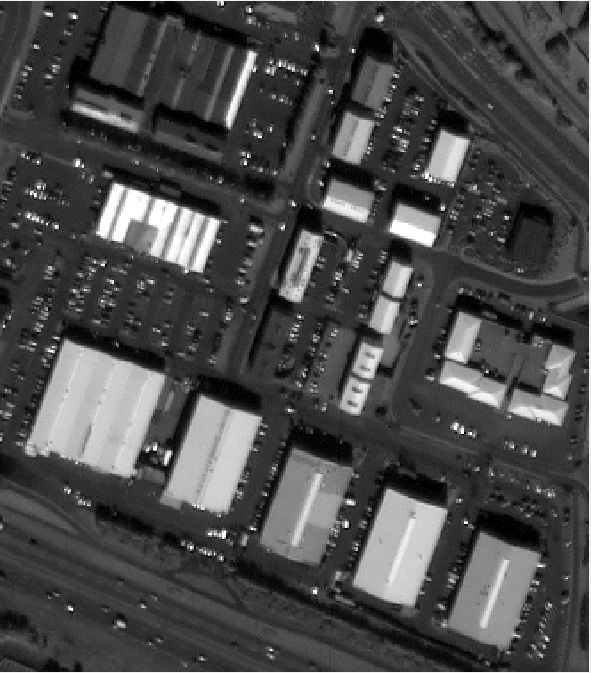}}
  \quad 
  \subfigure[]{
    \label{pb}
    \includegraphics[width=0.25\linewidth]{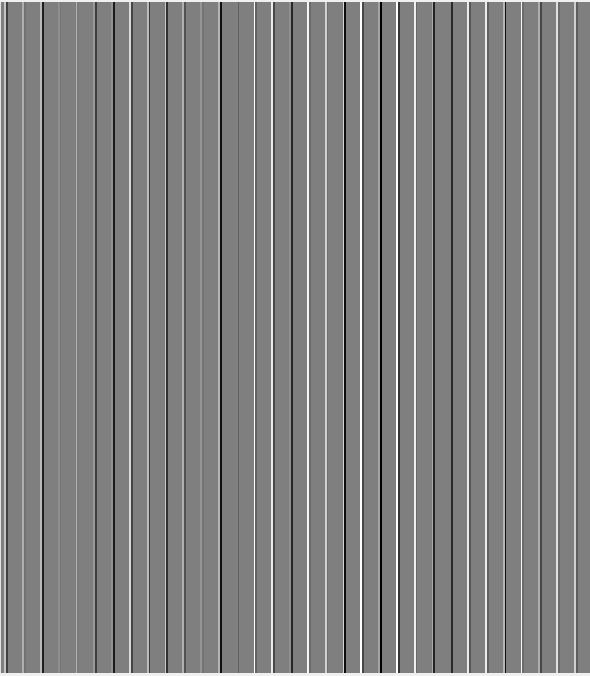}}
    \quad
  \subfigure[]{
	\label{pc}
	\includegraphics[width=0.25\linewidth]{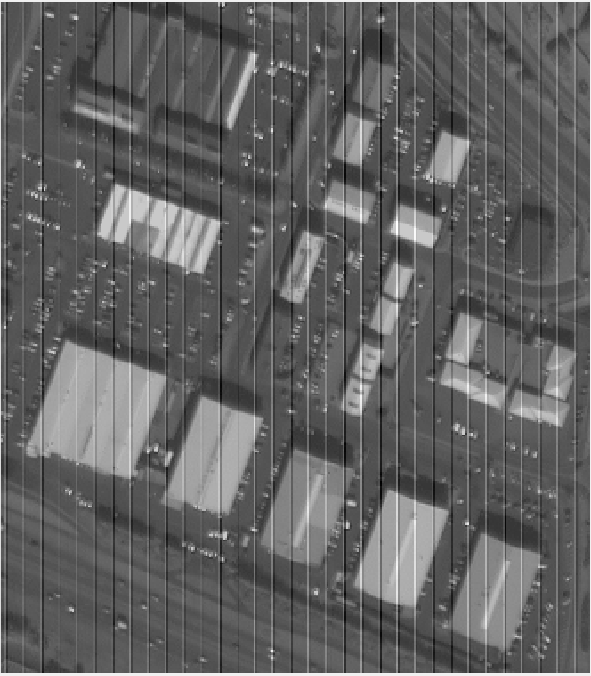}}\\
  \subfigure[]{
	\label{pd}
	\includegraphics[width=0.25\linewidth]{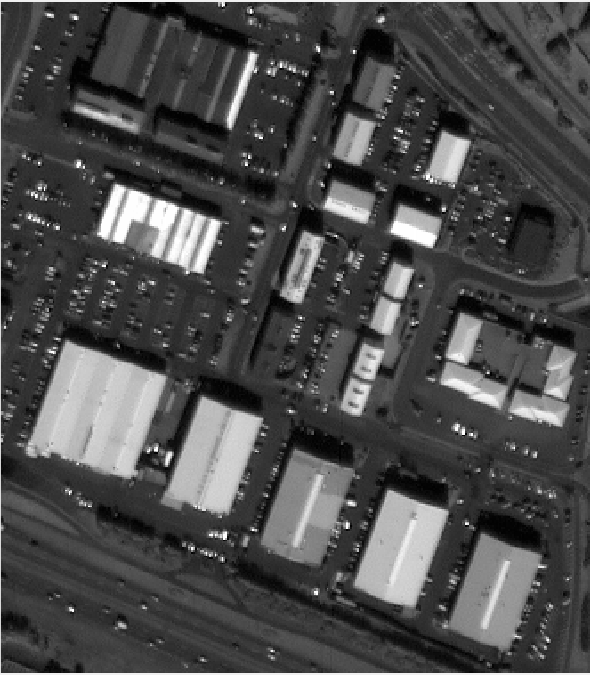}}
  \quad
  \subfigure[]{
    \label{pe}
    \includegraphics[width=0.25\linewidth]{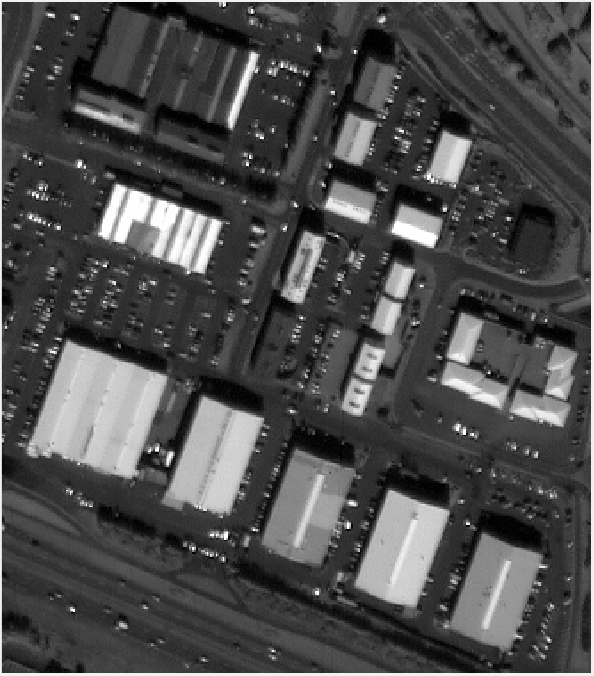}}
  \quad 
  \subfigure[]{
    \label{pf}
    \includegraphics[width=0.25\linewidth]{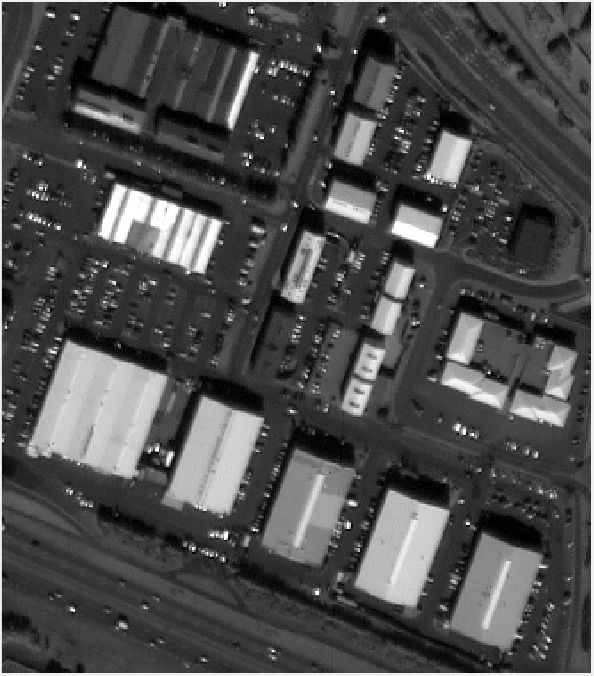}}
  \caption{The destripe results. (a) Original image; (b) stripe component; (c) degraded image; (d) estimated image by the convex model (pADMM); (e) estimated image by the convex model (dADMM); (f) estimated image by the nonconvex model}
  \label{pic_three}
\end{figure}

\section{Conclusion}
\label{sec:Conclusion}

In this paper, we have developed a nonconvex model for remote sensing stripe images. For the computational issues, we have applied the inexact PMM algorithm to solve the model with the inner subproblem solved by the dADMM. Meanwhile, we have analyzed the convergence of the algorithm based on the K{\L} property. Finally, numerical experiments demonstrate the superiority of the proposed model and algorithm.

\bibliographystyle{spmpsci}
\bibliography{DCforRemoteSensing}

\end{document}